\documentclass[12pt]{amsart}
\usepackage{amsmath, amssymb, amsthm}
\usepackage{enumerate}
\usepackage{bbm}

\makeatletter
\@namedef{subjclassname@2010}{%
\textup{2010} Mathematics Subject Classification}
\makeatother

\newtheorem{theorem}{Theorem}[section]

\newtheorem{lem}{Lemma}[section]
\numberwithin{equation}{section}
\newcommand{\indicator}[1]{\mathbbm{1}_{ {#1} }}

\title{On Local Median Oscillation Decompositions}
\author{Jonathan Poelhuis and Alberto Torchinsky}

\begin{document}
\begin{abstract}
In this note we generate two local median oscillation decompositions of an arbitrary measurable function and discuss some applications to Calder\'{o}n-Zygmund singular integral operators $T$.
These applications rely on the inequality $M^{\sharp}_{0,s}(Tf)(x) \leq c\,Mf(x)$, and we complete the results  given here with
 a  discussion  of a local version of this estimate.
\end{abstract}

\maketitle

The ``local mean oscillation'' decomposition of Lerner has proven particularly useful in recent literature. Such a functional decomposition was first considered in terms of averages by Garnett and Jones \cite{GarnettJones}, then suggested in terms of medians by Fujii \cite{Fujii1991}, and proved by Lerner \cite{Lerner2010}, \cite{LernerSummary}. In this note we generate two local median oscillation decompositions of an arbitrary measurable function and discuss some applications to Calder\'{o}n-Zygmund singular integral operators $T$.
These applications rely on the inequality $M^{\sharp}_{0,s}(Tf)(x) \leq c\,Mf(x)$
established in  \cite{JawerthTorchinsky}, and we complete the results  given here with
a local version of this estimate.

In what follows, we adopt the notations of \cite{MedContOsc} and \cite{Stromberg}. In particular, all cubes have sides parallel to the axes. Also, for a cube $Q \subset \mathbb{R}^n$ and $0 < t < 1$, we say that $m_f(t,Q) = \sup\{M : |\{y \in Q: f(y) < M\}| \leq t|Q|\}$ is the (maximal) median of $f$ over $Q$ with parameter $t$. For a cube $Q_0 \subset \mathbb{R}^n$ and $0 < s \leq 1/2$, the local sharp maximal function restricted to $Q_0$ of a measurable function $f$ at $x \in Q_0$ is
\[M^{\sharp}_{0,s,Q_0}f(x) = \sup_{x \in Q, Q \subset Q_0}\inf_c \; \inf \{\alpha \geq 0: |\{y \in Q: |f(y) - c| > \alpha\}| < s|Q|\}.\]


Additionally,
 we consider the maximal function $m^{t,\Delta}_{Q_0}$ defined as follows.
Let $\mathcal{D}$ be the set of dyadic cubes in $\mathbb{R}^n$. For a cube $Q \subset \mathbb{R}^n$, let $\mathcal{D}(Q)$ denote the family of dyadic subcubes relative to $Q$; that is to say, those formed by repeated dyadic subdivision of $Q$ into $2^n$ congruent subcubes.  Then
\[m^{t,\Delta}_{Q_0} f(x) = \sup_{x \in Q, Q \in \mathcal{D}(Q_0)} |m_f(t,Q)|.\]
A related non-dyadic maximal function was introduced by A. P. Calder\'{o}n  in order to exploit cancellation to obtain estimates for singular integrals
\cite{Calderon1972}.

Finally,   $\widehat{Q}$ denotes the dyadic parent of a cube $Q$.

\section{First Decomposition}   

Our first result holds for a range of indices, and so it extends Lerner's decomposition, which corresponds to the case $t=1/2$, $s=1/4$ in Theorem 1.1. On the other hand, the bound below is larger than his, but   it suffices for some applications. Also, the proof relies on medians and is somewhat more geometric.
 
\begin{theorem}
Let $f$ be a measurable function on a fixed cube $Q_0 \subset \mathbb{R}^n$, $0 < s < 1/2$, and $1/2 \leq t < 1-s$. Then there exists a (possibly empty) collection of subcubes $\{Q^v_j\} \subset \mathcal{D}(Q_0)$ and a family of collections
of indices $\{I^v_2\}_v$ such that
\begin{enumerate}
\item[\rm (i)]
for a.e. $\! x \in Q_0$,
\begin{align*}
|f(x) - m_f(t,Q_0)| &\leq 4M^{\sharp}_{0,s,Q_0}f(x)
\\
+ \sum_{v=1}^{\infty}&\sum_{j \in I^v_2} \Big (10n\inf_{y \in Q^v_j}M^{\sharp}_{0,s,\widehat{Q^v_j}}f(y) + 2 \inf_{y \in Q^v_j}M^{\sharp}_{0,s,Q^v_j}f(y) \Big )\indicator{Q^v_j}(x);
\end{align*}
\item[\rm (ii)]
for fixed $v$, the $\{Q^v_j\}$ are nonoverlapping;
\item[\rm(iii)]
if\, $\Omega^v = \bigcup_j Q^v_j$, then $\Omega^{v+1} \subset \Omega^v$; and,
\item[\rm (iv)]
for all $j$, $|\Omega^{v+1} \cap Q^v_j| \le \frac{s}{1-t} \, |Q^v_j|$.
\end{enumerate}
\end{theorem}

\begin{proof}
Let $E^1 = \{x \in Q_0: |f(x) - m_f(t,Q_0)| > 2\inf_{y \in Q_0}M^{\sharp}_{0,s,Q_0}f(y)\}$. If $|E^1|=0$, the decomposition halts -- trivially, for a.e. $\! x \in Q_0$,
\begin{equation}
|f(x) - m_f(t,Q_0)| \leq 2 \inf_{y \in Q_0}M^{\sharp}_{0,s,Q_0}f(y).
\end{equation}

So suppose that $|E^1| > 0$. Recall that by Lemma 4.1 in \cite{MedContOsc}, for $\eta > 0$,
\[ \big|\big\{x \in Q_0: |f(x) - m_f(t,Q_0)| \geq 2 \inf_{y \in Q_0}M^{\sharp}_{0,s,Q_0}f(y) + \eta\big\}\big| < s|Q_0|.\]
Thus, picking $\eta_k \to 0^+$, by continuity from below it readily follows that
\begin{equation}
\big|\big\{x \in Q_0: |f(x) - m_f(t,Q_0)| > 2 \inf_{y \in Q_0}M^{\sharp}_{0,s,Q_0}f(y)\big\}\big| \le s|Q_0|.
\end{equation}

Now let $f^0 = \left (f - m_f(t,Q_0) \right )\indicator{Q_0}$, and
\[\Omega^1 = \big\{x \in Q_0: m^{t,\Delta}_{Q_0}(f^0)(x) > 2\inf_{y \in Q_0}M^{\sharp}_{0,s,Q_0}f(y)\big\}.\]
Then by Theorem 2.1 in \cite{MedContOsc}, $E^1 \subset \Omega^1$ and $|\Omega^1| > 0$ as well. Write $\Omega^1 = \bigcup_j Q^1_j$ where the $Q^1_j$'s are nonoverlapping maximal dyadic subcubes of $Q_0$ such that
\begin{equation}
|m_{f^0}(t,Q^1_j)| > 2\inf_{y \in Q_0}M^{\sharp}_{0,s,Q_0}f(y), \quad {\text{and}}\quad  |m_{f^0}(t,\widehat{Q^1_j})| \leq 2\inf_{y \in Q_0}M^{\sharp}_{0,s,Q_0}f(y).
\end{equation}
Since $m_{f^0}(t,Q_0) = 0$, $Q^1_j \neq Q_0$ for any $j$.

Now since $t \geq 1/2$, from (1.10) in \cite{MedContOsc} it follows that
\begin{equation}
2\inf_{y \in Q_0}M^{\sharp}_{0,s,Q_0}f(y) < |m_{f^0}(t,Q^1_j)| \leq m_{|f^0|}(t,Q^1_j),
\end{equation}
and therefore by the definition of median
\begin{equation}
\big|\big\{x \in Q^1_j: |f^0(x)| > 2\inf_{y \in Q_0}M^{\sharp}_{0,s,Q_0}\big\}\big| \ge (1-t)|Q^1_j|.
\end{equation}
When these are summed, we have by (1.2) that
\begin{align*}
(1-t)\sum_j |Q^1_j| &\leq \sum_j \big|\big\{x \in Q^1_j: |f^0(x)| > 2\inf_{y \in Q_0}M^{\sharp}_{0,s,Q_0}f(y)\big\}\big|
\\
&\le \big|\big\{x \in Q_0: |f^0(x)| > 2\inf_{y \in Q_0}M^{\sharp}_{0,s,Q_0}f(y)\big\}\big| \le s|Q_0|,
\end{align*}
so that
\begin{equation}
\sum_j |Q^1_j| \leq \frac{s}{1-t}|Q_0|,
\end{equation}
where by the choice of $s$ and $t$, $s/(1-t) < 1$.

Let $\alpha^1_j = m_{f^0}(t,Q^1_j)$. By Lemma 4.3 in \cite{MedContOsc}, we see that
\begin{equation}
\big| m_{f^0}(t,Q^1_j) - m_{f^0}(t,\widehat{Q^1_j})\big| \le 10n\inf_{y \in Q^1_j}M^{\sharp}_{0,s,\widehat{Q^1_j}}f(y),
\end{equation}
and therefore by (1.3) and (1.7)
\begin{align}
|\alpha^1_j| &\leq \big|m_{f^0}(t,Q^1_j) - m_{f^0}(t, \widehat{Q^1_j})\big| + \big|m_{f_0}(t,\widehat{Q^1_j})\big|
\notag
\\
&\leq 10n\inf_{y \in Q^1_j}M^{\sharp}_{0,s,\widehat{Q^1_j}}f(y) + 2\inf_{y \in Q_0}M^{\sharp}_{0,s,Q_0}f(y).
\end{align}

The first iteration of the local median oscillation decomposition of $f$ when $|E^1| > 0$ is then as follows: for a.e. $\! x \in Q_0$, with $g^1 = f^0\indicator{Q_0 \setminus \Omega^1}$,
\[f^0(x) = g^1(x) + \sum_j \alpha^1_j \indicator{Q^1_j}(x) + \sum_j \big (f^0(x) - m_{f^0}(t,Q^1_j) \big )\indicator{Q^1_j}(x).\]
Note that $g^1$ has support off $\Omega^1$, and clearly for a.e. $\! x \in Q_0$,
\[|g^1(x)| \leq 2\inf_{y \in Q_0}M^{\sharp}_{0,s,Q_0}f(y).\]

Now focus on the second sum. Since $f^0(x) - m_{f^0}(t,Q) = f(x) - m_f(t,Q)$ for all cubes $Q$ and functions $f$ supported in $Q$, this sum equals
\[\sum_j \big( f(x) - m_f(t,Q^1_j) \big)\indicator{Q^1_j}(x).\]
The idea is to repeat the above argument for each function $f^1_j = \big( f - m_f(t,Q^1_j) \big)\indicator{Q^1_j}$, and so on.

We now describe the iteration. Assuming that $\{Q^{k-1}_j\}$ are the dyadic cubes corresponding to the $(k-1)$st generation of subcubes of $Q_0$ obtained as above, let
\[f_j^{k-1} = \big( f - m_f(t,Q^{k-1}_j) \big)\indicator{Q^{k-1}_j},\]
and
\[E^k_j = \big\{x \in Q^{k-1}_j: f^{k-1}_j(x) > 2\inf_{y \in Q^{k-1}_j}M^{\sharp}_{0,s,Q^{k-1}_j}f(y)\big\}.\]

If $|E^k_j| = 0$,  we write $s^k_j = f^{k-1}_j$ which satisfies, as in (1.1),
\begin{equation}
|s^k_j(x)| \leq 2\inf_{y \in Q^{k-1}_j}M^{\sharp}_{0,s,Q^{k-1}_j}f(y)
\end{equation}
for a.e. $\! x \in Q^{k-1}_j$. These are the ``$s$'' functions since the decomposition ``stops'' at $Q^{k-1}_j$; clearly $s^k_j$ has its support on $Q^{k-1}_j$, and $Q^{k-1}_j$ contains no further subcubes of the decomposition.

If $|E^k_j| > 0$, we define
\[\Omega^k_j = \big\{x \in Q^{k-1}_j: m^{t,\Delta}_{Q^{k-1}_j}f^{k-1}_j(x) > 2\inf_{y \in Q^{k-1}_j}M^{\sharp}_{0,s,Q^{k-1}_j}f(y)\big\} \supset E^k_j.\]
Note that the $Q^{k-1}_j$, and thus the $\Omega^k_j$, are nonoverlapping. Then $|\Omega^k_j| > 0$ as well, and
\[\Omega^k_j = \bigcup_i Q^k_i,\]
where the $Q^k_i$'s are nonoverlapping maximal dyadic subcubes of $Q^{k-1}_j$ such that
\begin{align}
|m_{f^{k-1}_j}(t,Q^k_i)| &> 2\inf_{y \in Q^{k-1}_j}M^{\sharp}_{0,s,Q^{k-1}_j}f(y), \quad {\text{and}}
\notag
\\
|m_{f^{k-1}_j}(t,\widehat{Q^k_i)}| &\leq 2\inf_{y \in Q^{k-1}_j}M^{\sharp}_{0,s,Q^{k-1}_j}f(y).
\end{align}
Then define
\[\Omega^k = \bigcup_j \Omega^k_j.\]

Let $\alpha^{k,j}_i = m_{f^{k-1}_j}(t,Q^k_i)$, and note that by (1.7) and (1.10)
\begin{align}
|\alpha^{k,j}_i| &\leq \big|m_{f^{k-1}_j}(t,Q^k_i) - m_{f^{k-1}_j}(t,\widehat{Q^k_i})\big| + \big|m_{f^{k-1}_j}(t,\widehat{Q^k_i})\big|
\notag
\\
&\leq 10n \inf_{y \in Q^k_i} M^{\sharp}_{0,s,\widehat{Q^k_i}}f(y) + 2\inf_{y \in Q^{k-1}_j}M^{\sharp}_{0,s,Q^{k-1}_j}f(y).
\end{align}

We then have
\[f^{k-1}_j(x) = g_j^k(x) + \sum_i \alpha^{k,j}_i\indicator{Q^k_i}(x) + \sum_i \big(f(x) - m_f(t,Q^k_i) \big)\indicator{Q^k_i}(x),\]
for a.e. $\! x \in Q^{k-1}_j$, where $g^k_j = f^{k-1}_j \indicator{Q^{k-1}_j \setminus \Omega^k_j}$ is readily seen to satisfy
\begin{equation}
|g^k_j(x)| \leq 2 \inf_{y \in Q^{k-1}_j}M^{\sharp}_{0,s,Q^{k-1}_j}f(y)
\end{equation}
for a.e. $\! x \in Q^{k-1}_j$. These are the ``$g$'' functions since the decomposition ``goes on'' or continues, into $Q^{k-1}_j$; $g^k_j$ has support on $Q^{k-1}_j$ away from $\Omega^k_j$, which are the next subcubes in the decomposition.

We separate the $Q^{k-1}_j$ into two families. One family, indexed by $I^k_1$, contains those cubes where the decomposition stops, and the other, indexed by $I^k_2$, where it continues. Specifically, let
\[I^k_1 = \{j : \Omega^k \cap Q^{k-1}_j = \emptyset\}, \quad I^k_2 = \{j: \Omega^k \cap Q^{k-1}_j \neq \emptyset\}.\]
Now we group the $Q^k_i$ based on which $Q^{k-1}_j$ contains them:  if $j \in I^k_2$, let
\[J_{j}^k = \{i : Q^k_i \subset Q^{k-1}_j\}.\]
These definitions 
give that
\[\Omega_{j}^k = \bigcup_{i \in J_{j}^k}Q^k_i.\]

Note that, as in (1.6),
\[|\Omega^k_j \cap Q^{k-1}_j| = \sum_{i \in J_{j}^k}|Q^k_i| \leq \frac{s}{1-t}|Q^{k-1}_j|\]
so that
\begin{align}
|\Omega^k| &= \sum_j |\Omega^k_j \cap Q^{k-1}_j| \leq \Big( \frac{s}{1-t} \Big)\sum_j |Q^{k-1}_j|
\notag
\\
&= \Big( \frac{s}{1-t} \Big)|\Omega^{k-1}| \leq \Big( \frac{s}{1-t} \Big)^k|Q_0|.
\end{align}

The $k$th iteration of the local median oscillation decomposition of the function $f$ is as follows: for a.e. $\! x \in Q_0$,

\[ f(x) - m_f(t,Q_0) = \sum_{v = 1}^k \Big(\sum_{j \in I^v_1}s^v_j + \sum_{j \in I^v_2} g^v_j \Big) + \sum_{v = 1}^k \sum_{j \in I^v_2} \sum_{i \in J_{j}^v}\alpha^{v,j}_i \indicator{Q^v_i}(x) + \psi^k(x),\]
where
\[\psi^k = \sum_{j \in I^k_2} \sum_{i \in J_{j}^k} \big(f - m_f(t,Q^k_i) \big )\indicator{Q^k_i}.\]

Since $\psi^k$ is supported in $\Omega^k$, by (1.13) it readily follows that $\psi^k \to 0$ a.e.\! in $Q_0$ as $k \to \infty$, and therefore
\begin{align*}
f(x) - m_f(t,Q_0) &= \sum_{v = 1}^{\infty} \Big( \sum_{j \in I^v_1}s^v_j + \sum_{j \in I^v_2} g^v_j \Big) + \sum_{v = 1}^{\infty}\sum_{j \in I^v_2} \sum_{i \in J^v_j} \alpha^{v,j}_i \indicator{Q^v_j}(x)
\\
&=S_1(x) + S_2(x),
\end{align*}
say.

In order to bound $|f(x) - m_f(t,Q_0)|$, consider first $S_1$. Of course, for all $v$ and $j$ the $s^v_j$'s have nonoverlapping support. This is also true for the $g^v_j$'s. Furthermore, the support of any $g^v_j$ is nonoverlapping with that of any $s^v_j$. So for every $v$, $j$, and a.e. $\! x \in Q_0$, by (1.9) and (1.12)
\begin{align}
\Big| \sum_{v=1}^{\infty} \Big(\sum_{j \in I^v_1}s^v_j + \sum_{j \in I^v_2} g^v_j \Big) \Big| &\leq \max \Big\{ \sup_{j \in I^v_1} \big\| f^{v-1}_j  \big\|_{L^{\infty}}, \; \sup_{j \in I^v_2} \big\| f^{v-1}_j\indicator{Q^{v-1}_j \setminus \Omega_{j}^v} \big\|_{L^{\infty}} \Big\}
\notag
\\
&\leq \max \Big\{ \sup_{j \in I^v_1} \Big(2\inf_{y \in Q^{v-1}_j} M^{\sharp}_{0,s,Q^{v-1}_j}f(y) \Big),
\notag
\\
&\qquad \qquad \qquad \qquad\quad \sup_{j \in I^v_2} \Big(2\inf_{y \in Q^{v-1}_j}M^{\sharp}_{0,s,Q^{v-1}_j}f(y) \Big) \Big\}
\notag
\\
&\leq 2\,M^{\sharp}_{0,s,Q_0}f(x).
\end{align}

We consider $S_2$ next. The summand for $v = 1$ is distinguished, so we deal with it separately. By (1.8) above,
\begin{align}
\Big |\sum_j \alpha_j^1 \indicator{Q^1_j}(x) \Big | &\leq \sum_j |\alpha_j^1| \indicator{Q^1_j}(x)
\notag
\\
&\leq \sum_j \Big (10n\inf_{y \in Q^1_j}M^{\sharp}_{0,s,\widehat{Q^1_j}}f(y) + 2\inf_{y \in Q_0}M^{\sharp}_{0,s,Q_0}f(y) \Big )\indicator{Q^1_j}(x)
\notag
\\
&\leq \sum_j \Big(10n \inf_{y \in Q^1_j}M^{\sharp}_{0,s,\widehat{Q^1_j}}f(y) \Big)\indicator{Q^1_j}(x) + 2\inf_{y \in Q_0}M^{\sharp}_{0,s,Q_0}f(y).
\end{align}

As for the other terms of the sum, by (1.11) we have
\begin{align}
\Big |\sum_{v = 2}^{\infty} &\sum_{j \in I^v_2} \sum_{i \in J^v_j}\alpha^{v,j}_i \indicator{Q^v_i}(x) \Big | \leq \sum_{v = 2}^{\infty} \sum_{j \in I^v_2} \sum_{i \in J_{j}^v}|\alpha^{v,j}_i| \indicator{Q^v_i}(x)
\notag
\\
&\leq \sum_{v = 2}^{\infty} \sum_{j \in I^v_2} \sum_{i \in J_{j}^v}\Big (10n \inf_{y \in Q^v_i} M^{\sharp}_{0,s,\widehat{Q^v_i}}f(y) + 2\inf_{y \in Q^{v-1}_j}M^{\sharp}_{0,s,Q^{v-1}_j}f(y) \Big )\indicator{Q^v_i}(x)
\notag
\\
&\leq \sum_{v = 2}^{\infty} \sum_{j \in I^v_2}\sum_{i \in J^v_j} \Big(10n \inf_{y \in Q^v_i}M^{\sharp}_{0,s,\widehat{Q^v_i}}f(y)
\Big )\indicator{Q^v_i}(x) \notag
\\
&\qquad\qquad\qquad +\sum_{v = 2}^{\infty}\sum_{j \in I^v_2} \Big(2\inf_{y \in Q^{v-1}_j}M^{\sharp}_{0,s,Q^{v-1}_j}f(y) \Big)\indicator{Q^{v-1}_j}(x).
\end{align}

We combine (1.15) and (1.16) and note that since the sum is infinite and the families $I^v_2$ are nested,
\begin{align}
\Big |\sum_{v = 1}^{\infty} &\sum_{j \in I^v_2} \sum_{i \in J^v_j}\alpha^{v,j}_i \indicator{Q^v_i}(x) \Big | \leq \sum_{v=1}^{\infty}\sum_{j \in I^v_2} \sum_{i \in J^v_j} \Big(10n\inf_{y \in Q^v_i}M^{\sharp}_{0,s,\widehat{Q^v_i}}f(y) \Big) \indicator{Q^v_i}(x)
\notag
\\
&\qquad\qquad\qquad\qquad\qquad\qquad + \sum_{v=1}^{\infty}\sum_{j \in I^v_2}\Big(2\inf_{y \in Q^{v-1}_j}M^{\sharp}_{0,s,Q^{v-1}_j}f(y) \Big)\indicator{Q^{v-1}_j}(x)
\notag
\\
&\leq \sum_{v=1}^{\infty}\sum_{j \in I^v_2} \Big(10n\inf_{y \in Q^v_j}M^{\sharp}_{0,s,\widehat{Q^v_j}}f(y) + 2 \inf_{y \in Q^v_j}M^{\sharp}_{0,s,Q^v_j}f(y)\Big) \indicator{Q^v_j}(x)
\notag
\\
&\qquad\qquad + 2\inf_{y \in Q_0}M^{\sharp}_{0,s,Q_0}f(y).
\end{align}

Combining (1.14) and (1.17), finally we get that for a.e. $\! x \in Q_0$,
\begin{align*}
|f(x) - &m_f(t,Q_0)|  \leq 4\, M^{\sharp}_{0,s,Q_0}f(x)
\\
&  \qquad + \sum_{v=1}^{\infty}\sum_{j \in I^v_2} \Big(10n\inf_{y \in Q^v_j}M^{\sharp}_{0,s,\widehat{Q^v_j}}f(y) + 2 \inf_{y \in Q^v_j}M^{\sharp}_{0,s,Q^v_j}f(y)\Big) \indicator{Q^v_j}(x).
\end{align*}
Thus the conclusion holds.
\end{proof}

\section{Applications}
Theorem 2.1  displays how the first decomposition leads to the
results   in \cite{Lerner2010} and \cite{LernerSummary}. Its proof is similar in spirit to Lerner's. In what follows, $M$ is the Hardy-Littlewood maximal function.

\begin{theorem}
For any weight $w$, $0 < s < 1/2$, $1/2 \leq t < 1-s$, $0 < \delta \leq 1$, and $f$ measurable on a cube $Q_0 \subset \mathbb{R}^n$,
\begin{align*}
\int_{Q_0}|f(x) - &m_f(t,Q_0)|\,w(x)\,dx
\\
&\leq c\int_{\mathbb{R}^n}\Big (M^{\sharp}_{0,s,Q_0}f(x) \Big )^{\delta}M\Big (\Big (M^{\sharp}_{0,s,Q_0}f(\cdot) \Big )^{1-\delta} w(\cdot) \Big)(x)\,dx,
\end{align*}
where $c = c_{n,s,t}$. If $f$ is measurable on $\mathbb{R}^n$ such that $m_f(t,Q_0) \to 0$ as $Q_0 \to \mathbb{R}^n$, then
\[\int_{\mathbb{R}^n}|f(x)|\,w(x)\,dx \leq c\int_{\mathbb{R}^n}\Big (M^{\sharp}_{0,s}f(x) \Big )^{\delta}M\Big (\Big (M^{\sharp}_{0,s}f(\cdot) \Big )^{1-\delta} w(\cdot) \Big )(x)\,dx.\]
\end{theorem}

\begin{proof}
Applying Theorem 1.1, we have
\begin{align*}
\int_{Q_0}|&f(x) - m_f(t,Q_0)|\,w(x)\,dx \leq 4\int_{Q_0}M^{\sharp}_{0,s,Q_0}f(x) \, w(x)\,dx
\\
&  \quad +\sum_{v=1}^{\infty}\sum_{j \in I^v_2} \Big(10n\inf_{y \in Q^v_j}M^{\sharp}_{0,s,\widehat{Q^v_j}}f(y) + 2 \inf_{y \in Q^v_j}M^{\sharp}_{0,s,Q^v_j}f(y) \Big)\int_{Q^v_j}w(x)\,dx.
\end{align*}
Now it easily follows that
\begin{align*}
\int_{Q_0}M^{\sharp}_{0,s,Q_0}f(x) & \, w(x)\,dx 
 \\
&\le\int_{\mathbb{R}^n} \Big (M^{\sharp}_{0,s,Q_0}f(x) \Big )^{\delta} M \Big ( \Big (M^{\sharp}_{0,s,Q_0}f(\cdot) \Big )^{1-\delta}w(\cdot) \Big )(x)\,dx.
\end{align*}

By construction, the $Q^v_j$'s are nonoverlapping over fixed $v$, but not over all $v$; indeed, each $Q^v_j$ is a subcube of some $Q^{v-1}_i$. So we define $F^v_j = Q^v_j \setminus \Omega^{v+1}$, which are pairwise disjoint over all $v$ and $j$. And since
\[|\Omega^{v+1} \cap Q^v_j| \leq \frac{s}{1-t}|Q^v_j|,\]
we know that
\[|F^v_j| \geq \Big (1 - \frac{s}{1-t} \Big )|Q^v_j| = c_{s,t}|Q^v_j|.\]
We can then compute
\begin{align*}
\inf_{y \in Q^v_j} M^{\sharp}_{0,s,\widehat{Q^v_j}}f(y) \int_{Q^v_j} & w(x)\,dx\\
& \le c_{s,t}\frac{|F^v_j|}{|Q^v_j|} \inf_{y \in Q^v_j}M^{\sharp}_{0,s,Q_0}f(y) \int_{Q^v_j}w(x)\,dx
\\
&= c_{s,t} \Big (\inf_{y \in Q^v_j}M^{\sharp}_{0,s,Q_0}f(y)\Big )^{\delta} |F^v_j| \frac{1}{|Q^v_j|}\int_{Q^v_j}\Big (M^{\sharp}_{0,s,Q_0}f(x) \Big )^{1-\delta}w(x)\,dx
\\
&\leq c_{s,t} \Big( \int_{F^v_j} \left (M^{\sharp}_{0,s,Q_0}f(x) \right )^{\delta}dx \Big) \inf_{y \in F^v_j}M \Big( \Big
 (M^{\sharp}_{0,s,Q_0}f(\cdot) \Big)^{1-\delta}w(\cdot) \Big)(y).
\end{align*}
Summing, we see that
\begin{align*}
\sum_{v=1}^{\infty} \sum_{j \in I^v_2} \inf_{y \in Q^v_j}&M^{\sharp}_{0,s,\widehat{Q^v_j}}f(y) \int_{Q^v_j}w(x)dx
\\
&\le c_{s,t} \int_{\mathbb{R}^n} \left (M^{\sharp}_{0,s,Q_0}f(x) \right )^{\delta} M \left ( \left (M^{\sharp}_{0,s,Q_0}f(\cdot) \right )^{1-\delta}w(\cdot) \right )(x)\,dx.
\end{align*}

We bound $\sum_{v=1}^{\infty}\sum_{j \in I^v_2} \inf_{y \in Q^v_j}M^{\sharp}_{0,s,Q^v_j}f(y)$ similarly.
\end{proof}


\section{Revisiting $M^{\sharp}_{0,s}(Tf)(x) \leq c M f(x)$}

Lerner's application of Theorem 1.2 is to Calder\'{o}n-Zygmund singular integral operators $T$ via the inequality $M^{\sharp}_{0,s}(Tf)(x) \leq c\,Mf(x)$. We prove here a local version of this estimate.
$M^{\sharp}$ denotes the sharp maximal function.

\begin{theorem}
 Let $T$ be a Calder\'{o}n-Zygmund singular integral operator defined
by  \[Tf(x) = {\rm{p.v.}} \int_{\mathbb{R}^n} k(x,y)f(y)dy\]
such that
\begin{enumerate}
\item[\rm (1)]
for some $C, \delta > 0$, 
$k(x,y)$ satisfies
\[ | k(x,y) - k(x',y)| \leq C\frac{|x - x'|^{\delta}}{|x - y|^{n+\delta}}\]
     whenever $x,x'\in Q$ and $y\in (2Q)^c$ for some cube $Q$, and
\item[\rm (2)]
$T$ is of weak-type (1,1).
\end{enumerate}
Then for  $0 < s \leq 1/2$,  any cube $Q_0$,  and $x\in Q_0$,
\[M^{\sharp}_{0,s,Q_0}(Tf)(x) \leq c \sup_{x \in Q,\, Q \subset Q_0} \inf_{y \in Q}Mf(y).\]

Moreover, if we also have that $T(1)=0$,
then
\[M^{\sharp}_{0,s,Q_0}(Tf)(x) \leq c \sup_{x \in Q, Q \subset Q_0} \inf _{y \in Q} M^{\sharp} f(y).\]

In particular, if $Q_0 = \mathbb{R}^n$, then $M^{\sharp}_{0,s}(Tf)(x) \leq c Mf(x)$ and
 $M^{\sharp}_{0,s}(Tf)(x)$ $ \leq c M^{\sharp}f(x)$,
 respectively.
\end{theorem}
\begin{proof}

We prove the case when $T(1) =0$. Fix a cube $Q_0 \subset \mathbb{R}^n$  and take $x \in Q_0$. Let $Q \subset Q_0$ be a cube containing $x$ with center $x_Q$ and sidelength $l_Q$. Let $1/2 \leq t \leq 1-s$,  $f_1 = (f - m_f(t,Q)) \indicator{2Q}$, and $f_2 = (f - m_f(t,Q))\indicator{(2Q)^c}$.
Then by the linearity of $T$, $Tf(z)- Tf_2(x_Q)= Tf_1(z)+ Tf_2(z)-Tf_2(x_Q)$ for $z\in Q$.


We claim that
\begin{equation}
\|Tf_2 - Tf_2(x_Q)\|_{L^{\infty}(Q)} \leq c_{2} \inf_{y \in Q}M^{\sharp}f(y),
\end{equation}
and
\begin{equation}
|\{z \in Q: |Tf_1(z)| > \lambda\}| < s\,|Q|
\end{equation}
for $\lambda = c \inf_{y \in Q}M^{\sharp}f(y)$, where $c$ is appropriately chosen.

We prove (3.1) first. For any $z \in Q$,
\begin{align}
|Tf_2(z) - Tf_2(x_Q)| &\leq \int_{(2Q)^c}|f(y) - m_f(t,Q)||k(z,y) - k(x_Q,y)|dy
\notag
\\
&\leq c_{n,\delta} l_Q^{\delta} \int_{(2Q)^c} \frac{|f(y) - m_f(t,Q)|}{|z - y|^{n+\delta}}dy
\notag
\\
&\leq c_{n,\delta} l_Q^{\delta} \sum_{m=1}^{\infty} \int_{2^mQ} \frac{|f(y) - m_f(t,Q)|}{|z - y|^{n+\delta}} dy
\notag
\\
&\leq c_{n,\delta}l_Q^{\delta} \sum_{m=1}^{\infty} \frac{1}{(2^m l_Q)^{n+\delta}} \int_{2^mQ} |f(y) - m_f(t,Q)|dy
\notag
\\
&\leq c_{n,\delta} \sum_{m=1}^{\infty} \frac{1}{2^{m\delta}}\frac{1}{|2^{m}Q|} \int_{2^mQ} |f(y) - m_f(t,Q)|dy.
\end{align}

It readily follows from Proposition 1.1 in \cite{MedContOsc} that for any cube $Q$,
 \begin{align}
\frac{1}{|Q|}\int_{Q} |f(y) - m_f(t,Q)|\, dy &\le
\frac{1}{|Q|}\int_{Q} |f(y) -  f_Q|\, dy + |f_Q- m_f(t,Q)|\notag
\\
&\le 2\inf_{y \in Q}M^{\sharp}f(y) + \frac{1}{s}\frac{1}{|Q|}\int_{Q} |f(y) -  f_Q|\, dy\notag
\\
&\le c_s  \inf_{y \in Q}M^{\sharp}f(y) .
\end{align}

Also, for all $j$, by (3.4),
\begin{align}
|m_f(t,2^j Q) - m_f(t,2^{j-1} Q)| &\leq
 m_{|f-m_f(t,2^{j} Q)|} (t,2^{j-1} Q)
\notag
\\
& \leq  \frac{1}{s} \frac{1}{|2^{j-1}Q|} \int_{2^j Q} |f(y) - m_f(t,2^j Q)|\, dy\notag\\
&\leq c_{n,s} \inf_{y \in Q}M^{\sharp}f(y).
\end{align}

Then (3.4) and (3.5) give that
\begin{align}
\int_{2^m Q} |f(y) - m_f(t,Q)|dy &\leq \int_{2^m Q} |f(y) - m_f(t,2^m Q)|\, dy
\notag
\\
&\qquad\quad + \int_{2^m Q} \sum_{j=1}^m |m_f(t,2^j Q) - m_f(t, 2^{j-1}Q)|\, dy
\notag
\\
&\leq  |2^{m} Q|\, c_s \inf_{y \in Q}M^{\sharp}f(y)
\notag
\\
&\qquad\quad + |2^{m} Q|\, c_{n,s} \sum_{j=1}^m \inf_{y \in Q}M^{\sharp}f(y)
\notag
\\
&\leq |2^{m} Q|\, c_{n,s} (1+m) \inf_{y \in Q}M^{\sharp}f(y).
\end{align}

Using (3.6), we can bound (3.3) as
\begin{align*}
|Tf_2(z) - Tf_2(x_Q)| &\leq c_{n,\delta,s} \sum_{m=1}^{\infty} \frac{1}{2^{m\delta}} (1+m) \inf_{y \in Q}M^{\sharp}f(y)
\\
&
\leq c_{n,\delta,s} \inf_{y \in Q}M^{\sharp}f(y),
\end{align*}
and so
\begin{equation}
\|Tf_2 - Tf_2(x_Q)\|_{L^{\infty}(Q_0)} \leq c_{2} \inf_{y \in Q}M^{\sharp}f(y)\,.
\end{equation}

As for $(3.2)$, since $T$ is of weak-type (1,1), by (3.4), (3.5), and Lemma 4.3 in \cite{MedContOsc}
 we have that for any $\lambda > 0$,
\begin{align}
|\{z \in Q: |Tf_1(z)| > \lambda\}| &\leq \frac{C}{\lambda} \int_{2Q}|f(y) - m_f(t,Q)|dy
\notag
\\
&\leq \frac{C}{\lambda}\int_{2Q}|f(y) - m_f(t,2Q)|dy
\notag
\\
&\qquad \qquad + C
\frac{|2Q|}{\lambda} |m_f(t,Q) - m_f(t,2Q)|
\notag
\\
&\leq  c_{n,s}\frac{|Q|}{\lambda} \inf_{y \in 2Q} M^{\sharp}f(y) +  c_{n,s}' \frac{|Q|}{\lambda} \inf_{y \in Q}M^{\sharp}_{0,s}f(y)
\notag
\\
&\leq c_1 \frac{ |Q|}{\lambda} \inf_{y \in Q}M^{\sharp}f(y).
\end{align}

Choose $\lambda = c \inf_{y \in Q}M^{\sharp}f(y)$, where $c > \max \{c_2, c_1/s\}$. Then
\[ |\{z \in Q : |Tf_1(z)| > c \inf_{y \in Q}M^{\sharp}f(y)\}| < s|Q|.\]

Then (3.1) and (3.2) give 
\begin{align*}
|\{z \in Q&: |Tf(z) - Tf_2(x_Q)| > 2c \inf_{y \in Q}M^{\sharp}f(y)\}|
\\
&\le |\{z \in Q: |Tf_2(z) - Tf_2(x_Q)| > c \inf_{y \in Q}M^{\sharp}f(y)\}|
\\
&\qquad\qquad + |\{z \in Q: |Tf_1(z)| > c \inf_{y \in Q}M^{\sharp}f(y)\}|
\\
&<{s}|Q|.
\end{align*}
Whence
\[\inf_{c'} \; \inf \{\alpha \geq 0: |\{z \in Q: |Tf(z) - {c'}| > \alpha\}| < s|Q|\}\le c \inf_{y \in Q}M^{\sharp}f(y),\]
and consequently,
\[M^{\sharp}_{0,s,Q_0}Tf(x)\le c \sup_{x\in Q,\, Q\subset Q_0} \inf_{y \in Q}M^{\sharp}f(y).\]

To prove the case where $T(1) \neq 0$, let $f_1 = f\indicator{2Q}$ and $f_2 = f\indicator{(2Q)^c}$, and proceed as above.
\end{proof}


\section{Second Decomposition}
A different local median oscillation decomposition is needed for other applications, including Lerner's proof of the $A_2$ conjecture.

\begin{theorem}
Let $f$ be a measurable function on a fixed cube $Q_0 \subset \mathbb{R}^n$, $0 < s < 1/2$, and $1/2 \leq t < 1-s$. Then there exists a (possibly empty) collection of subcubes $Q^v_j \in \mathcal{D}(Q_0)$ and a family of collections of indices $\{I^v_2\}_v$ such that

\begin{enumerate}
\item[\rm (i)]
for a.e. $\! x \in Q_0$,
\begin{align*}
|&f(x) - m_f(t,Q_0)| \leq 8\,M^{\sharp}_{0,s,Q_0}f(x) +
\\
&\sum_{v=1}^{\infty}\sum_j \Big ( m_{|f - m_f(t,\widehat{Q^v_j})|}\Big (1-(1-t)/2^n,\widehat{Q^v_j} \Big ) + m_{|f - m_f(t,Q^{v}_j)|}(t,Q^{v}_j) \Big )\indicator{Q^v_j}(x);
\end{align*}
\item[\rm (ii)]
for fixed $v$, $\{Q^v_j\}$ are pairwise nonoverlapping families;
\item[\rm (iii)]
if $\Omega^v = \bigcup_j Q^v_j$, then $\Omega^{v+1} \subset \Omega^v$; and
\item[\rm (iv)]
for all $j$, $|\Omega^{v+1} \cap Q^v_j| \leq \frac{s}{1-t}|Q^v_j|$.
\end{enumerate}
\end{theorem}

\begin{proof}
We follow the proof of Theorem 1.1 in form, with 
a few definitional changes. First note that for any cube $Q$,
\begin{equation}
m_{|f - m_f(t,Q)|}(t,Q) \leq  4\inf_{y \in Q}M^{\sharp}_{0,s,Q}f(y).
\end{equation}
To see this, from Proposition 1.1 and (4.3) in \cite{MedContOsc},
\begin{align*}
m_{|f - m_f(t,Q)|}(t,Q) &\leq m_{|f - m_f(1-s,Q)| + |m_f(1-s,Q) - m_f(t,Q)|}(t,Q)
\\
&\leq 2m_{|f - m_f(1-s,Q)|}(t,Q) \leq 2m_{|f - m_f(1-s,Q)|}(1-s,Q)
\\
&\leq 4\inf_{y \in Q}M^{\sharp}_{0,s,Q}f(y).
\end{align*}

We define $E^1 = \{x \in Q_0: |f(x) - m_f(t,Q_0)| > m_{|f - m_f(t,Q_0)|}(t,Q_0)\}$. If $|E^1| = 0$, the decomposition halts, just as in Theorem 1.1. So we suppose $|E^1| > 0$. We then define
\[\Omega^1 = \{x \in Q_0: m^{t,\Delta}_{Q_0}(f^0) > m_{|f - m_f(t,Q_0)|}(t,Q_0)\}.\]

Proceeding as above, we have that $\Omega^1 = \bigcup_j Q^1_j$ so that (as in (1.3))
\begin{equation}
|m_{f^0}(t,Q^1_j)| > m_{|f^0|}(t,Q_0), \quad {\text{and}}\quad |m_f(t, \widehat{Q^1_j})| \leq m_{|f^0|}(t,Q_0).
\end{equation}
Furthermore, we also have that
\begin{equation}
\sum_j |Q^1_j| \leq \frac{s}{1-t}|Q_0|.
\end{equation}

Before continuing, observe that
\begin{equation}
m_f(t,Q) \leq m_f \big(1-(1-t)/2^n,\widehat{Q} \big).
\end{equation}
To see this, note that
\begin{align*}
|\{y \in \widehat{Q}: f(y) \geq m_f(t,Q)\}| &\geq |\{y \in Q: f(y) \geq m_f(t,Q)\}|
\\
&\geq (1-t)|Q| = \frac{1-t}{2^n}\big| \widehat{Q} \big|,
\end{align*}
so taking complements in $\widehat{Q}$ we have
\[|\{y \in \widehat{Q}: f(y) < m_f(t,Q)\}| \leq \Big ( 1 - \frac{1-t}{2^n} \Big )\big|\widehat{Q}\big|.\]
Note also that by our choice of $t$, $1 - (1-t)/2^n \geq 1/2$.

Let $\alpha^1_j = m_{f^0}(t,Q^1_j)$. By (4.4) we have
\begin{align}
|\alpha^1_j| &\leq |m_{f^0}(t,Q^1_j) - m_{f^0}(t, \widehat{Q^1_j})| + |m_{f^0}(t,\widehat{Q^1_j})|
\notag
\\
&\leq m_{|f^0 - m_{f^0}(t,\widehat{Q^1_j})|}(t,Q^1_j) + m_{|f^0|}(t,Q_0)
\notag
\\
&\leq m_{|f - m_f(t,\widehat{Q^1_j})|}\Big (1-(1-t)/2^n,\widehat{Q^1_j} \Big) + m_{|f^0|}(t,Q_0).
\end{align}

This gives the first iteration of the 
decomposition of $f$ when $|E^1| > 0$: for a.e. $\! x \in Q_0$, with $g^1 = f^0 \indicator{Q_0 \setminus \Omega^1}$,
\[f^0(x) = g^1(x) + \sum_j \alpha^1_j \indicator{Q^1_j}(x) + \sum_j \big ( f^0(x) - m_{f^0}(t,Q^1_j) \big)\indicator{Q^1_j}(x).\]
Clearly by (4.1)
\[|g^1(x)| \leq m_{|f - m_f(t,Q_0)|}(t,Q_0) \leq 4\inf_{y \in Q_0}M^{\sharp}_{0,s,Q_0}f(y) \leq 4M^{\sharp}_{0,s,Q_0}f(x)\]
a.e.\! on $Q_0 \setminus \Omega^1$.

By proceeding as in the proof of Theorem 1.1, the result follows.
\end{proof}

\section{Other Applications}
     Lerner's local mean oscillation decomposition also plays a prominent role in a recent paper of Cruz-Uribe, P\'{e}rez, and Martell \cite
     {CruzUribePerez2012}.  We follow the argument in their paper and indicate
      the role of the local median oscillation decomposition of Theorem 4.1   in their proofs.


The setting is as  follows. For $Q \in \mathcal{D}$,
    let $h_Q$ be the Haar function on $Q$.
Given an integer $\tau \geq 0$, a Haar shift operator of index $\tau$ is an operator of the form
\[H_{\tau}f(x) = \sum_{Q \in \mathcal{D}} \sum_{\substack{Q', Q'' \in \mathcal{D}(Q)\\2^{-\tau n}|Q| \leq |Q'|, |Q''|}}a_{Q',Q''} \langle f, h_{Q'} \rangle h_{Q''}(x),\]
where $a_{Q',Q''}$ is a constant such that
\[ |a_{Q',Q''}| \leq C \Big ( \frac{|Q'|}{|Q|} \frac{|Q''|}{|Q|} \Big )^{1/2}.\]
Finally, for a cube $Q$, $Q^{\tau}$ is the $\tau$-th generation ancestor of $Q$, i.e. the unique dyadic cube containing $Q$ so that $|Q^{\tau}| = 2^{\tau n}|Q|$

  Lemma 4.2 in \cite{CruzUribePerez2012} is key in proving their main result. It makes use of rearrangements  and it can be rephrased in terms of
    medians as follows. $M^d$ and $M^{\sharp,d}_{0,s,Q_0}$ denote the dyadic Hardy-Littlewood maximal operator and the dyadic local sharp maximal
    operator restricted to $Q_0$, respectively.


\begin{lem}
Let $0 < s \leq 1/2$. Then for any measurable function $f$, a dyadic cube $Q_0$, and  $x \in Q_0$,
\begin{align}
\inf_c m_{|H_{\tau}f - c|}(s,Q_0) &\leq C_{\tau,n,s} \frac{1}{|Q^{\tau}_0|} \int_{Q^{\tau}_0} |f(x)| dx
\\
M^{\sharp,d}_{0,s,Q_0}(H_{\tau}f)(x) &\leq C_{\tau,n,s} M^d f(x).
\end{align}
\end{lem}

(5.1) follows by Proposition 1.2 in \cite{MedContOsc}, and (5.2) is an immediate consequence of (5.1).


  We   need one further inequality. Note  that for any cube $Q$ and any constant $c$, since $t \leq 1 - (1-t)/2^n$,
\begin{align*}
m_{|f - m_f(t,Q)|}(1-(1-t)/2^n,Q) &\leq m_{|f - c|}(1-(1-t)/2^n,Q) + m_{|f - c|}(t,Q)
\\
&\leq 2 m_{|f - c|}(1-(1-t)/2^n,Q).
\end{align*}
Thus
\begin{equation}
m_{|f - m_f(t,Q)|}(1-(1-t)/2^n,Q) \leq 2\inf_c m_{|f - c|}(1-(1-t)/2^n,Q).
\end{equation}

  Returning to the  proof, the use of the local mean oscillation decomposition occurs at line (5.1) of \cite{CruzUribePerez2012}. In its place
    we apply
    the dyadic version of Theorem 4.1. Thus, for any $0 < s < 1/2$ and $1/2 \leq t < 1-s$, by (5.1) and (5.3),
\begin{align*}
|&H_{\tau}f(x) - m_{H_{\tau}f}(Q_N)| \leq 8M^{\sharp,d}_{0,s,Q_0}(H_{\tau}f)(x)
\\
&\qquad \qquad \qquad + \sum_{v=1}^{\infty}\sum_j \Big ( m_{|H_{\tau}f - m_{H_{\tau}f}(t,\widehat{Q^v_j})|}\Big (1-(1-t)/2^n,\widehat{Q^v_j} \Big )
\\
& \qquad \qquad \qquad \qquad + m_{|H_{\tau}f - m_f(t,Q^{v}_j)|}(t,Q^{v}_j) \Big )\indicator{Q^v_j}(x)
\\
&\leq C_{\tau,n,s}M^df(x)+ \sum_{v=1}^{\infty} \sum_j \Big( C_{\tau,n,t} \frac{1}{|\widehat{(Q^v_j)}^{\tau}|}\int_{\widehat{Q^v_j}^{\tau}} |f(x)| dx
\\
&\qquad \qquad \qquad + C_{\tau,n,t}\frac{1}{|(Q^v_j)^{\tau}|} \int_{(Q^v_j)^{\tau}} |f(x)|dx \Big)\indicator{Q^v_j}(x)
\\
&\leq C_{\tau,n,s}Mf(x) + C_{\tau,n,t}\sum_{v=1}^{\infty} \sum_j \Big ( \frac{1}{|\widehat{(Q^v_j)}^{\tau}|}\int_{\widehat{(Q^v_j)}^{\tau}} |f(x)| dx \Big )\indicator{Q^v_j}(x).
\end{align*}
The computation then continues as in \cite{CruzUribePerez2012}.

   In a similar fashion, using Theorem 4.1 and computations
   analogous to  those in \cite{LernerA2}, the interested reader can parallel Lerner's proof of the $A_2$ conjecture
      using medians in place of rearrangements.

\bibliographystyle{amsplain}
\bibliography{myrefs}

\end{document}